\documentclass[11pt,reqno]{amsart}
\usepackage{tikz-cd}
\usepackage{amsfonts, amssymb, amscd}
\usepackage{amsmath}
\usepackage{verbatim}
\usepackage{amssymb}
\usepackage{mathrsfs}
\usepackage{graphicx}
\usepackage{cite}
\usepackage[all]{xy}
\usepackage{tikz}
\usepackage{subfiles}
\usepackage[toc,page]{appendix}
\usepackage{filecontents}
\usepackage{tikz-cd}
\usepackage{graphicx}
\usepackage{hyperref}
\usepackage{float}
\usepackage{color}
\usetikzlibrary{automata,arrows.meta}
\usepackage{spectralsequences}

\def\gr {{\operatorname{gr}}}

\def\im {{\operatorname{im}}}
\def\id {{\operatorname{id}}}
\def\rk {{\operatorname{rk}}}
\def\C {{\mathbb C}}

\def\R {{\mathbb R}}
\def\Z {{\mathbb Z}}
\def\PP {{\mathbb P}}

\allowdisplaybreaks[3]

\usetikzlibrary{arrows.meta,decorations.markings}

\tikzset{
    midarrow/.style={
    postaction={decorate},
    decoration={markings, mark=at position #1 with {\arrow{Stealth}}}
    },
    midarrow/.default=0.55
}

\theoremstyle{definition}

\newtheorem{theorem}{Theorem}[section]
\newtheorem{lemma}[theorem]{Lemma}
\newtheorem{proposition}[theorem]{Proposition}
\newtheorem{definition}[theorem]{Definition}
\newtheorem{notation}[theorem]{Notation}

\newtheorem{corollary}[theorem]{Corollary}
\newtheorem{remark}[theorem]{Remark}

\newtheorem{question}[theorem]{Question}

\numberwithin{equation}{section}

\begin{document}

\title{Minimal resolutions of toric substacks by line bundles}

\begin{abstract}
We construct minimal resolutions of pushforwards of structure sheaves of toric substacks of smooth toric stacks by line bundles, by realizing them as strong deformation retracts of the cellular resolutions constructed by Hanlon, Hicks, and Lazarev. We also provide a canonical and combinatorial description of the differentials of these minimal resolutions. The construction combines the homological perturbation lemma with Moore-Penrose inverses, and suggests a general strategy for extracting minimal resolutions from explicit non-minimal cellular resolutions, with potential broader applications to other settings in commutative algebra.
\end{abstract}

\author{Zengrui Han}
\address{Department of Mathematics\\
University of Maryland\\
College Park, MD 20742} \email{zhan223@umd.edu}

\maketitle
\tableofcontents

\section{Introduction}

\subsection{Motivation}

Resolutions of diagonals play an important role in the study of algebraic varieties. In the case of toric varieties and stacks, this picture is particularly well understood. On the level of cohomology, Fulton and Sturmfels \cite{Fulton-Sturmfels} gave a combinatorial formula for the cohomology classes of toric subvarieties in toric varieties in terms of the fan displacement rule, which in turn has applications to toric mirror symmetry \cite{BHduality,Han}. On the level of complexes or derived categories, several related constructions of toric diagonal resolutions, as well as their variants and applications, have appeared in the works of Brown-Erman \cite{BE2}, Brown-Sayrafi \cite{BS}, Anderson \cite{Anderson}, Favero-Huang \cite{FH}, and many others. More recently, inspired by homological mirror symmetry for toric varieties \cite{Bondal,FLTZ}, Hanlon, Hicks and Lazarev \cite{HHL} constructed canonical cellular resolutions of pushforwards of structure sheaves of toric substacks of smooth toric stacks by using a certain class of line bundles on the ambient toric stack, namely the Bondal-Thomsen collection. Their work also inspired subsequent work of Ballard et al. \cite{BBB+} on King's conjecture and the Cox category. Later, Borisov and the author \cite{BH} provided an equivalent formulation and a simplified proof of their result; see also \cite{BCHSY} for a generalization in the context of commutative algebra.

\smallskip

In general, however, the Hanlon-Hicks-Lazarev resolution (HHL resolution for short) can be far from minimal. There are algorithmic methods for constructing minimal resolutions from a given non-minimal resolution by iteratively removing redundant summands, for example using tools from discrete or algebraic Morse theory. Previously, Favero and Sapronov \cite{FS} gave a topological formula for the Betti numbers of such minimal resolutions by using tools from homological mirror symmetry for toric varieties. However, obtaining a canonical and purely combinatorial construction of the minimal resolution itself remains a difficult problem.

\smallskip

The goal of this paper is to address this problem. We give a canonical and combinatorial construction of the minimal resolution of pushforwards of structure sheaves of toric substacks of smooth toric stacks by line bundles in the Bondal-Thomsen collection. More importantly, we provide a combinatorial description of the differentials in this minimal resolution. Thus our result can be seen as an optimal refinement of the construction in \cite{HHL}. It is also worth mentioning that our proof does not rely on any form of mirror symmetry. As a corollary, our construction yields a purely algebraic proof of the Betti number formula of Favero-Sapronov.

\subsection{Main construction}
We briefly review the construction of \cite{HHL}; see Section \ref{sec:Bondal-HHL} for details. Let $\mathcal{Y}\hookrightarrow\mathcal{X}$ be a closed embedding of toric stacks of codimension $k$. Denote the associated Bondal stratification on the compact torus $(S^1)^k$ by $S$; roughly speaking, $S$ is a cell decomposition of $(S^1)^k$ induced by the rays in the fan of the ambient toric stack, see Definition \ref{def:Bondal-stratification} for details. Hanlon, Hicks and Lazarev \cite{HHL} constructed a cellular resolution of the pushforward of the structure sheaf of $\mathcal{Y}$ by using the Bondal-Thomsen collection of $\mathcal{X}$:
    \begin{align}\label{eq:HHL}
        C_{\bullet}:\ 0\rightarrow\bigoplus_{\substack{\dim\sigma=k }} \mathcal{O}_{\mathcal{X}}(\sigma) \xrightarrow{d} \cdots \rightarrow \bigoplus_{\substack{\dim\sigma=1}} \mathcal{O}_{\mathcal{X}}(\sigma) \xrightarrow{d} \bigoplus_{\substack{\dim\sigma=0}} \mathcal{O}_{\mathcal{X}}(\sigma)\rightarrow0
    \end{align}
Roughly speaking, they associate a monomial weight $\mathrm{HHL}_{\sigma,\tau}$ to each pair of cells $(\sigma,\tau)$ with $\dim\sigma=\dim\tau+1$ according to the values of the linear functions that are used to define the Bondal stratification on these cells, and use these to define the differentials $d$.

\smallskip

The combinatorics of our construction is more complicated. We need to separate the morphisms that appear in the HHL resolution into two classes, depending on whether their corresponding line bundles are the same or not.

\smallskip

{\bf Type I pairs:} $(\sigma,\tau)$ with $\mathcal{O}(\sigma) > \mathcal{O}(\tau)$ under the partial order on the Bondal-Thomsen collection, i.e., $\operatorname{Hom}(\mathcal{O}(\tau), \mathcal{O}(\sigma))=0$. We define the weight to be the usual $\mathrm{HHL}_{\sigma,\tau}$ appearing in the differential of the HHL resolution and call it the \textit{HHL weight}. 

\smallskip

{\bf Type II pairs:} $(\sigma,\tau)$ with $\mathcal{O}(\sigma) = \mathcal{O}(\tau)$. For pairs of this type, first recall that the cells in the Bondal stratification are naturally partitioned into subsets according to the associated line bundles in the Bondal-Thomsen collection, which we denote by $S=\sqcup_{[a]\in BT}S_{[a]}$. For each fixed $[a]\in BT$, the cells in $S_{[a]}$ form a natural cellular chain complex of $\C$-vector spaces
\begin{align}\label{eq:BMcomplex-in-introduction}
    0\rightarrow \bigoplus_{\substack{\sigma\in S_{[a]} \\ \dim\sigma = k}}\C\sigma \rightarrow \cdots \rightarrow \bigoplus_{\substack{\sigma\in S_{[a]} \\ \dim\sigma = 1}}\C\sigma \rightarrow \bigoplus_{\substack{\sigma\in S_{[a]} \\ \dim\sigma = 0}}\C\sigma \rightarrow 0
\end{align}
that computes the Borel-Moore homology of the piece $S_{[a]}$. A type II pair $(\sigma,\tau)$ will naturally appear in such a sequence, and we define the \textit{Moore-Penrose weight} (\textit{MP weight} for short) $\mathrm{MP}_{\sigma,\tau}$ by the Moore-Penrose inverse of the differentials of these cellular chain complexes, see Definition \ref{def:MP-weights} for the precise definition.

\medskip

Now let $\sigma,\tau$ be two strata in the Bondal stratification $S$ with $\dim\sigma = \dim\tau + 1$ and $\mathcal{O}(\sigma)\geq \mathcal{O}(\tau)$ in the Bondal-Thomsen poset $BT$. We consider all paths $\lambda$ from $\sigma$ to $\tau$ of the form
    \[
\begin{tikzpicture}[>=stealth]
\matrix (m) [matrix of math nodes,
  row sep=2.8em,
  column sep=2.8em,
  ampersand replacement=\&]{
  \sigma=\sigma_0 \& \& \sigma_1 \& \& \cdots \& \sigma_m \& \\
  \& \tau_0 \& \& \tau_1 \& \& \& \tau_m=\tau \\
};

\draw[->, blue] (m-1-1) -- node[midway, above, sloped, font=\tiny]
  {\color{blue}$\mathrm{HHL}_{\sigma_0,\tau_0}$} (m-2-2);
\draw[->,red] (m-2-2) -- node[midway, above, sloped, font=\tiny]
  {\color{red}$\mathrm{MP}_{\sigma_1,\tau_0}$} (m-1-3);
\draw[->, blue] (m-1-3) -- node[midway, above, sloped, font=\tiny]
  {\color{blue}$\mathrm{HHL}_{\sigma_1,\tau_1}$} (m-2-4);
\draw[->,red] (m-2-4) -- node[midway, above, sloped, font=\tiny]
  {\color{red}$\mathrm{MP}_{\sigma_2,\tau_1}$} (m-1-5);
\draw[->, blue] (m-1-6) -- node[midway, above, sloped, font=\tiny]
  {\color{blue}$\mathrm{HHL}_{\sigma_m,\tau_m}$} (m-2-7);
\end{tikzpicture}
\]
    where each $(\sigma_i,\tau_i)$ is a type I pair (blue), and each $(\sigma_{i+1},\tau_i)$ is a type II pair (red). The weight of such a path is defined as the signed product of all weights along it, where the sign depends on the length of the path. This allows us to define a collection of modified morphisms on the HHL resolution:
\begin{align}\label{eq:modifiedHHL}
        0\rightarrow\bigoplus_{\substack{\dim\sigma=k }} \mathcal{O}_{\mathcal{X}}(\sigma) \xrightarrow{\Sigma} \cdots \rightarrow \bigoplus_{\substack{\dim\sigma=1}} \mathcal{O}_{\mathcal{X}}(\sigma) \xrightarrow{\Sigma} \bigoplus_{\substack{\dim\sigma=0}} \mathcal{O}_{\mathcal{X}}(\sigma)\rightarrow0
    \end{align}
where $\Sigma = \oplus \Sigma_{\sigma,\tau}$ and $\Sigma_{\sigma,\tau}$ is defined by the sum of weights over all paths from $\sigma$ to $\tau$. Finally, combinatorial Hodge decomposition induced by the differentials in \eqref{eq:BMcomplex-in-introduction} allows us to project \eqref{eq:modifiedHHL} onto its harmonic part
\begin{align}\label{eq:min-res-introduction}
    0\rightarrow\bigoplus_{\substack{[a]\in BT }} \mathcal{O}_{\mathcal{X}}(a)^{\oplus\beta_{k,a}} \xrightarrow{d_{\min}} \cdots  \xrightarrow{d_{\min}} \bigoplus_{\substack{[a]\in BT }} \mathcal{O}_{\mathcal{X}}(a)^{\oplus\beta_{0,a}}\rightarrow0
\end{align}
where $\beta_{i,a} = \dim H^{\mathrm{BM}}_{i}(S_{[a]},\C)$, and the differential $d_{\min}$ is obtained by restricting $\Sigma$ to the harmonic part. The main result of this paper is the following.

\begin{theorem}[= Theorem \ref{thm:main}]
    The complex \eqref{eq:min-res-introduction} is the minimal resolution of $i_*\mathcal{O}_{\mathcal{Y}}$.
\end{theorem}

\smallskip

\subsection{Virtual resolutions of modules over the Cox ring}

By \cite{BH,BCHSY}, the HHL resolutions admit an equivalent formulation in terms of graded virtual resolutions over the Cox ring. For simplicity, we restrict in this discussion to the case of toric varieties. This perspective may be of particular interest to readers with a background in commutative algebra. For previous results in this direction, see, for example, \cite{BES, BE, BCHS, CH, Yang}.

\smallskip

Let $S$ be the Cox ring\footnote{Here, and only in this paragraph, $S$ denotes the Cox ring. This should not be confused with the Bondal stratification denoted by $S$ above. Since the Cox ring does not appear in the main body of the paper, no ambiguity should arise.} of $X$. Recall that the Cox ring $S$ is naturally graded by the divisor class group $\mathrm{Cl}(X)$. Under the correspondence between $\mathrm{Cl}(X)$-graded $S$-modules and coherent sheaves on $X$, the HHL resolution \eqref{eq:HHL} can be seen as a \textit{virtual resolution} of the $\mathrm{Cl}(X)$-graded $S$-module $S/I_{Y}$, where $I_Y$ is the defining ideal of the toric subvariety $Y\hookrightarrow X$. Similarly, the minimal resolution \eqref{eq:min-res-introduction} can be seen as a virtual resolution
\begin{align}\label{eq:res-S-modules}
    0\rightarrow\bigoplus_{\substack{[a]\in BT }} S(a)^{\oplus\beta_{k,a}} \xrightarrow{d_{\min}} \cdots  \xrightarrow{d_{\min}} \bigoplus_{\substack{[a]\in BT }} S(a)^{\oplus\beta_{0,a}}\rightarrow0
\end{align}
of graded $S$-modules. In this language, the main result of this paper can be formulated as follows.

\begin{theorem}
    Let $Y\hookrightarrow X$ be a normal closed toric subvariety of a simplicial projective toric variety. Denote the defining ideal of $Y$ by $I_Y$. Then the complex \eqref{eq:res-S-modules} is a minimal virtual resolution of the $\mathrm{Cl}(X)$-graded $S$-module $S/I_Y$.
\end{theorem}

\smallskip

\subsection{Strategy of the proof}

We briefly explain the strategy of the proof. One of the main tools used in this paper is the homological perturbation lemma. Roughly speaking, it addresses the following question.

\begin{question}
Let $(H_{\bullet},d_H)$ and $(C_{\bullet},d_C)$ be chain complexes. Suppose that $(H_{\bullet},d_H)$ is a strong deformation retract of $(C_{\bullet},d_C)$, with homotopy map $h$ on $(C_{\bullet},d_C)$. If we introduce a ``small'' perturbation $\delta$ of the differential $d_C$, so that the perturbed differential is $d_{C,\infty}=d_C+\delta$, can one modify the differential $d_H$ accordingly so as to obtain a new complex $(H_{\bullet},d_{H,\infty})$ which is again a strong deformation retract of the perturbed complex $(C_{\bullet},d_{C,\infty})$?
\end{question}

\smallskip

The homological perturbation lemma gives an explicit formula for the perturbed differential $d_{H,\infty}$ in terms of the original data.

\smallskip

In our situation, $(C_{\bullet},d_C)$ is the associated graded complex of the HHL resolution with respect to the natural filtration induced by the poset structure on the Bondal-Thomsen collection, and $(H_{\bullet},d_H)$ is its homology complex with trivial differentials. We then consider the perturbation given by the difference between the original HHL differential and the differential on the associated graded complex. With this choice, the perturbed complex $(C_{\bullet},d_{C,\infty})$ is precisely the HHL resolution, and the resulting perturbed retract $(H_{\bullet},d_{H,\infty})$ is the desired minimal resolution.

\smallskip

There is, however, an important point in applying the homological perturbation lemma: we have not yet specified the homotopy map $h$ on $(C_{\bullet},d_C)$. In principle, the output of the perturbation lemma depends on this choice. Thus a key point is to find a canonical homotopy contracting the associated graded complex of the HHL resolution onto its homology complex. The Moore-Penrose inverse provides exactly such a choice. The advantage of using the Moore-Penrose inverse is that it does not require any auxiliary choices. As a result, the construction in this paper is canonical in the sense that the differentials are completely determined by the Bondal stratification itself. Moreover, the combinatorial nature of the Moore-Penrose inverse allows us to give a purely combinatorial formula for these differentials.

\smallskip

The homotopy $h$ used to construct the strong deformation retract is not unique. If one drops the requirement of canonicity, it is sometimes possible to choose homotopy maps of a simpler form, for example when a perfect Morse matching on the Bondal stratification is available. See the discussion at the beginning of Section \ref{sec:construction} and Remark \ref{rmk:choice-of-homotopy-2} for general discussions on alternative choices of $h$, and see the end of Section~\ref{sec:example} for an explicit example.

\smallskip

Our use of Moore-Penrose inverses is strongly inspired by the work of Eagon-Miller-Ordog \cite{EMO}. It was brought to the author's attention by Ezra Miller that the homological perturbation lemma also appears in the work of Amelotte and Briggs \cite{AB23}.

\subsection{Potential applications to commutative algebra}

Minimal free resolutions have been studied extensively in commutative algebra in various settings; see, for example, \cite{Tchernev,AB23,AB25,LMO}. It seems plausible that the approach developed in this paper, namely, combining the homological perturbation lemma with Moore-Penrose inverses, or more generally with suitable splittings or homotopies, can be applied in broader settings. For instance, starting from any kinds of cellular resolutions, one may hope to use the perturbation lemma together with Moore-Penrose inverses to minimize them while preserving the underlying combinatorics of the cellular resolutions. This perspective may be especially useful for resolutions constructed from LCM lattices, Betti posets, or related combinatorial data. We plan to investigate this direction in future work.

\subsection{Organization of the paper}

The paper is organized as follows. In Section \ref{sec:Bondal-HHL} we review the construction of the Bondal stratification and the Hanlon-Hicks-Lazarev resolution. In Section \ref{sec:MP-HPL} we review the theory of Moore-Penrose inverses and the homological perturbation lemma. In Section \ref{sec:construction} we apply the homological perturbation lemma to construct a strong deformation retract of the HHL resolution, and show that it is indeed a minimal resolution. In Section \ref{sec:differentials} we give a combinatorial description of the differentials in this minimal resolution. Finally, in Section \ref{sec:example} we give an explicit example to illustrate how to compute the minimal resolutions in practice.

\smallskip

{\bf Acknowledgements.} This work was greatly inspired by the author's joint paper \cite{BH} with Lev Borisov, and the author is deeply grateful to him for his consistent support and warm encouragement over the years. The author would like to thank Andrew Hanlon for bringing this question to his attention, for a careful reading of an earlier draft, and for valuable comments and suggestions. The author is also grateful to Mykola Sapronov for a conversation at the 2026 AMS Spring Eastern Sectional Meeting in Boston, which led the author to the work of Eagon, Miller, and Ordog \cite{EMO}. The author would further like to thank Ezra Miller for valuable suggestions that greatly improved the exposition of the paper and for pointing out useful references. Finally, the author would like to thank Daniel Erman, Mykola Sapronov and Junyan Zhao for careful reading of the paper and helpful conversations.

\section{Bondal stratification and Hanlon-Hicks-Lazarev resolutions}\label{sec:Bondal-HHL}

In this section we review the construction of the Bondal stratification and the Hanlon-Hicks-Lazarev resolution associated to a closed embedding of toric stacks. Our notation follows \cite[\S 3]{BH}, where the construction is reformulated in terms of homogeneous coordinates.

\smallskip

Rather than starting from the original stacky-fan description, we fix a quadruple
\begin{align*}
    (L,\Sigma,\psi:L\rightarrow\Lambda,G)
\end{align*}
where $L\cong\Z^n$ is a lattice with a chosen basis $v_1,\cdots,v_n$, $\Sigma$ is a subfan of the first orthant fan in \(L_{\mathbb{R}}\), $\psi:L\rightarrow \Lambda$ is a lattice morphism to a rank-$k$ lattice $\Lambda$ with finite cokernel, and $G$ is an algebraic group that maps to $\mathrm{ker}\,(L\otimes\C^* \rightarrow \Lambda\otimes\C^*)$.

\smallskip

Let $U_{\Sigma}\subset \C^n$ be the open subset consisting of points whose set of zero coordinates lies in a cone of $\Sigma$, and set $\mathcal{X}:=\mathcal{U}_{\Sigma,G}:=[U_{\Sigma}/G]$. We denote by \(D_i\) the torus-invariant divisor on \(\mathcal{X}\) cut out by the equation \(x_i=0\).

\smallskip

This description is related to the original HHL setup as follows. Suppose that $i:\mathcal{Y}\hookrightarrow\mathcal{X}$ is a closed embedding of toric stacks induced by a diagram
\begin{align*}
    \xymatrix{
        L_Y \ar@{^(->}[r] \ar[d]^{\beta_Y} & L_X \ar[d]^{\beta_X} & \\
        N_Y \ar@{^(->}[r]^{\phi} & N_X 
    }
\end{align*}
where the horizontal morphisms are embeddings of lattices, and the fan $\Sigma_{Y}$ is obtained by intersecting cones of $\Sigma_X$ with the subspace $(L_Y)_{\mathbb{R}}$. 

\smallskip

Following \cite[\S 3]{BH}, we pass to the Cox-style data by setting
\begin{align*}
    L=\mathbb{Z}^{\Sigma_X(1)},\qquad
\widetilde{\Sigma}=\{\mathrm{Cone}(e_\rho:\rho\in\sigma)\mid \sigma\in \Sigma_X\},\qquad
\Lambda=\operatorname{coker}\phi,
\end{align*}
and defining $\psi$ to be the composition
\begin{align*}
    L \rightarrow L_X \xrightarrow{\beta_X} N_X \rightarrow \operatorname{coker}\phi=\Lambda.
\end{align*}
Moreover,
\begin{align*}
    G:=\mathrm{ker}(L\otimes\C^* \to N_X\otimes\C^*)
\end{align*}
acts on $U_{\widetilde{\Sigma}}$, and the HHL complex attached to the original embedding $i$ is precisely the HHL complex associated to the input data $(L,\Sigma,\psi,G)$.

\smallskip

With this notation, we define the real torus
\begin{align*}
    T:=\operatorname{Hom}(\Lambda,\mathbb{R}/\mathbb{Z})
      =\Lambda_{\mathbb{R}}^*/\Lambda^*
      \cong (S^1)^k
\end{align*}
and its universal cover
\begin{align*}
    \Lambda_{\mathbb{R}}^*:=\operatorname{Hom}(\Lambda,\mathbb{R})\cong \mathbb{R}^k.
\end{align*}

\smallskip

\begin{definition}[Bondal stratification]\label{def:Bondal-stratification}
    Each basis vector $v_i$ of $L$ defines a linear map
\begin{align*}
    H_{i}:\ \Lambda^*_{\R} \rightarrow\R,\quad f\mapsto f(\psi(v_i)).
\end{align*}
They induce a periodic hyperplane arrangement defined by $H_i = a$ for $a\in \Z$ on $\Lambda^*_{\R}$, which descends to the torus $T$. We denote the induced stratification on $\Lambda^*_{\R}$ and $T$ by $\widetilde{S}$ and $S$ respectively, and call them the \textit{Bondal stratifications}.
\end{definition}

For a stratum $\widetilde{\sigma}$ in $\widetilde{S}$ and a point $x$ in the relative interior of $\widetilde{\sigma}$, the integer $\lceil H_i(x)\rceil$ is independent of the choice of $x$. We denote this integer by $\lceil H_i(\widetilde{\sigma})\rceil$.

\begin{definition}
    Let $\sigma \in S$ be a stratum in the Bondal stratification of the real torus $T$. Let $\widetilde{\sigma}\in \widetilde{S}$ be an arbitrary lift of $\sigma$. We define
    \begin{align*}
        \mathcal{O}_{\mathcal{X}}(\sigma):=\mathcal{O}_{\mathcal{X}}\left( - \sum_{i=1}^n \left\lceil H_{i}(\widetilde{\sigma})  \right\rceil D_{i} \right).
    \end{align*}
    Note that the isomorphism class of the bundle does not depend on the choice of the lift $\widetilde{\sigma}$. We call the set of isomorphism classes of such line bundles  the \textit{Bondal-Thomsen collection} of $\mathcal{X}$, and denote it by $BT$.
\end{definition}

\begin{notation}
For simplicity we will denote a torus-invariant line bundle on $\mathcal{X}$ by $\mathcal{O}(a)$ where $a\in\Z^n$. The isomorphism class of $\mathcal{O}(a)$ only depends on the image $[a]$ of $a$ under the linear map $\psi: L\rightarrow\Lambda$.
\end{notation}

\begin{definition}
    We denote the set of cells $\sigma$ in the Bondal stratification $S$ whose associated line bundle $\mathcal{O}(\sigma)$ has isomorphism class $[a]\in BT$ by $S_{[a]}$.
\end{definition}

Now we define morphisms $\mathcal{O}(\sigma)\rightarrow\mathcal{O}(\tau)$. We fix an arbitrary choice of orientations for all strata in the Bondal stratification on the torus $T$. Note that this also induces orientations for all lifts of the strata in the universal cover $\Lambda^*_{\R}$. Let $\sigma$ be an $m$-dimensional stratum in $S$, and $\tau$ an $(m-1)$-dimensional facet of $\sigma$. We take an arbitrary lift $\widetilde{\sigma}$ of $\sigma$ in the universal cover $\Lambda^*_{\R}$, and look at all facets of $\widetilde{\sigma}$ that are mapped to $\tau$ under the quotient map. Let $\widetilde{\tau}$ be such a facet. We denote $\epsilon_{i}:=\left\lceil H_{i}(\widetilde{\sigma})  \right\rceil - \left\lceil H_{i}(\widetilde{\tau})  \right\rceil $ and define
\begin{align*}
    \mathcal{O}(\sigma)\xrightarrow{\mathrm{sgn}(\widetilde{\sigma},\widetilde{\tau})\prod x_{i}^{\epsilon_{i}}}\mathcal{O}(\tau)
\end{align*}
where $\mathrm{sgn}(\widetilde{\sigma},\widetilde{\tau})$ equals $1$ if the orientations on $\widetilde{\sigma}$ and $\widetilde{\tau}$ are compatible and $-1$ otherwise.

\smallskip

\begin{definition}
    We call the complex of line bundles on $\mathcal{X}$ 
    \begin{align}\label{eq-HHL-complex}
        C_{\bullet}:\ 0\rightarrow\bigoplus_{\substack{\dim\sigma=k }} \mathcal{O}_{\mathcal{X}}(\sigma) \xrightarrow{d} \cdots \xrightarrow{d} \bigoplus_{\substack{\dim\sigma=1}} \mathcal{O}_{\mathcal{X}}(\sigma) \xrightarrow{d} \bigoplus_{\substack{\dim\sigma=0}} \mathcal{O}_{\mathcal{X}}(\sigma)\rightarrow0
    \end{align}
    the \textit{HHL complex} associated to the closed embedding $i:\mathcal{Y}\hookrightarrow \mathcal{X}$, where the differentials $d$ are given by the direct sum of all $\mathcal{O}(\sigma)\rightarrow\mathcal{O}(\tau)$ defined as above.
\end{definition}

\smallskip

The main result of \cite{HHL} is the following.
\begin{theorem}\label{thm:HHL}
    The HHL complex is a locally free resolution of the pushforward $i_*\mathcal{O}_{\mathcal{Y}}$.
\end{theorem}

\smallskip

\section{Moore-Penrose inverses and homological perturbation lemma}\label{sec:MP-HPL}

\subsection{Moore-Penrose inverses}\label{subsec:MP-inverses}

In this section we briefly review the theory of Moore-Penrose inverses and their combinatorics. For a more detailed treatment, we refer the reader to \cite{EMO}. Here we will only include the minimal amount of material that is necessary for our later applications.

\smallskip

\begin{definition}
    Let $\partial:\C^n\rightarrow\C^m$ be a linear map, where we denote the standard bases of $\C^n$ and $\C^m$ by $\{e_i\}$ and $\{f_j\}$ respectively. We define a \textit{hedge} $ST$ for $\partial$ to be a pair $(S,T)$, where $S\subseteq\{f_i\}$ and $T\subseteq\{e_i\}$ with $|S| = |T| = \rk(\partial)$, such that the following compositions
\begin{align*}
    \C\{T\} \hookrightarrow \C^n \xrightarrow{\partial} \im\partial,\qquad \im\partial\hookrightarrow \C^m \twoheadrightarrow \C\{S\}
\end{align*}
are isomorphisms, where $\C\{T\}$ and $\C\{S\}$ denote the subspaces spanned by the corresponding subsets of basis vectors. We call $S$ a \textit{stake set}, and call $T$ a \textit{shrubbery}. Given a hedge $ST$ for $\partial$, we define
\begin{align*}
    \partial_{S\times T}:\ \C\{T\} \rightarrow \C\{S\}
\end{align*}
to be the restriction of $\partial$ to $\C\{T\}\subseteq \C^n$ and $\C\{S\}\subseteq \C^m$. Clearly $\partial_{S\times T}$ is invertible.
\end{definition}

\smallskip

\begin{definition}
    Let $\partial:\C^n\rightarrow\C^m$ be a linear map and fix a hedge $ST$ for $\partial$. Denote the complement of $S$ in $\{f_i\}$ by $\overline{S}$. We define the \textit{hedge splitting} 
    \begin{align*}
        \partial_{ST}^+: \C^m\rightarrow\C^n
    \end{align*} 
    by its action on the basis $\overline{S}\cup\partial T$:
    \begin{itemize}
        \item $\partial_{ST}^+(\sigma) = 0$ for any $\sigma\in\overline{S}$.
        \item $\partial_{ST}^+(\partial \tau) = \tau$ for any $\tau\in T$.
    \end{itemize}
\end{definition}

\smallskip

\begin{definition}
    The \textit{Moore-Penrose inverse} $\partial^+$ of $\partial$ is defined to be the unique linear map $\C^m\rightarrow\C^n$ satisfying the following four conditions:
    \begin{align*}
        \partial\partial^+\partial = \partial,\quad \partial^+\partial\partial^+ = \partial^+,\quad (\partial\partial^+)^* = \partial\partial^+,\quad (\partial^+\partial)^* = \partial^+\partial
    \end{align*}
    where $*$ is the conjugate transpose.
\end{definition}

\begin{remark}
    Equivalently, the Moore-Penrose inverse $\partial^+$ of $\partial$ is the unique linear map $\C^m\rightarrow\C^n$ such that (1) it inverts $\partial$ on $(\ker \partial)^{\perp}$; (2) it is zero on $(\im \partial)^{\perp}$.
\end{remark}

\smallskip

The next theorem allows one to compute the Moore-Penrose inverse as the sum of hedge splittings over all hedges for the linear map.

\smallskip

\begin{theorem}[Hedge Formula]
    The Moore-Penrose inverse $\partial^+$ of $\partial$ is given by the following formula
    \begin{align*}
        \partial^+ = \frac{1}{\Delta}\sum_{ST}|\det\partial_{S\times T}|^2 \partial_{ST}^+.
    \end{align*}
    where $\Delta$ is defined as $\sum_{ST}|\det\partial_{S\times T}|^2$, and the sums are taken over all hedges $ST$ for $\partial$.
\end{theorem}
\begin{proof}
    See \cite[Corollary 5.9]{EMO}.
\end{proof}

For later use, we need the following formula for the entries of the Moore-Penrose inverse $\partial^+$.

\begin{proposition}\label{prop:MP-inverse-entry}
    Let $\partial:\C^n\rightarrow\C^m$ be a linear map and $\partial^+$ be its Moore-Penrose inverse. Then the entry of $\partial^+$ at the $i$-th row and $j$-th column is given by the following formula
    \begin{align*}
        \partial^+_{ij} = \frac{1}{\Delta}\sum_{\substack{ ST \text{ hedge} \\ j\in S, i\in T}} (-1)^{a+b} \overline{\det(\partial_{S\times T})}\det(\partial_{S\backslash\{j\}\times T\backslash\{i\}}).
    \end{align*}
    where $a$ and $b$ are the positions of $i$ in $T$ and $j$ in $S$ respectively.
\end{proposition}
\begin{proof}
    Denote the standard basis vectors of $\C^n$ and $\C^m$ by $\{e_i\}$ and $\{f_j\}$ respectively. Then the entry $\partial^+_{ij}$ is given by
    \begin{align*}
        \partial^+_{ij} = \langle e_i, \partial^+(f_j) \rangle = \frac{1}{\Delta}\sum_{ST}|\det\partial_{S\times T}|^2 \langle e_i, \partial_{ST}^+(f_j) \rangle
    \end{align*}
    according to the hedge formula. Now we claim
    \begin{align*}
        \langle e_i, \partial_{ST}^+(f_j) \rangle = (\partial_{S\times T}^{-1})_{\mathrm{pos}_T(i),\mathrm{pos}_S(j)}
    \end{align*}
    if $j\in S$ and $i\in T$ and is zero otherwise, where $\mathrm{pos}_T(i)$ and $\mathrm{pos}_S(j)$ are the positions of $i$ in $T$ and $j$ in $S$ respectively. To see this, first note that if $j\not\in S$ then by definition $\partial_{ST}^+(f_j)=0$, hence we assume $j\in S$. Since $\partial_{S\times T}:\C\{T\}\rightarrow\C\{S\}$ is an isomorphism, there exists a unique $x_j\in \C\{T\}$ such that $\partial_{S\times T}(x_j)=f_j$. In matrix form, we have
    \begin{align*}
        x_j = \sum_{i\in T}(\partial_{S\times T}^{-1})_{\mathrm{pos}_T(i),\mathrm{pos}_S(j)} e_i.
    \end{align*}
    Consider $f_j = (f_j - \partial x_j) + \partial x_j$. By construction the first term is in $\C\{\overline{S}\}$, and the second in the span of $\partial(T)$. Applying $\partial_{ST}^{+}$ we get
    \begin{align*}
        \partial_{ST}^{+}(f_j) = 0 + \partial_{ST}^{+}(\partial x_j)= x_j = \sum_{i\in T}(\partial_{S\times T}^{-1})_{\mathrm{pos}_T(i),\mathrm{pos}_S(j)} e_i
    \end{align*}
    and the claim follows. The proposition now follows from the adjugate matrix formula
    \begin{align*}
        (\partial_{S\times T}^{-1})_{\mathrm{pos}_T(i),\mathrm{pos}_S(j)} = (-1)^{a+b} \frac{\det\left( \partial_{S\backslash\{j\} \times T\backslash\{i\}} \right)}{\det\left( \partial_{S\times T} \right)}.
    \end{align*}
\end{proof}

\subsection{Homological perturbation lemma}\label{subsec:HPL}

In this section, we review some basic facts about the homological perturbation lemma. All materials are standard, but we include a detailed statement and proof for the sake of completeness. The proof provided here is inspired by \cite{HPL} and is completely formal. See also \cite{CLZ}.

\smallskip

\begin{definition}
    A \textit{strong deformation retract datum} (SDR datum)
    \[
\begin{tikzpicture}[>=Stealth,baseline=(current bounding box.center)]
  \node (L) at (0,0) {$ (H_{\bullet}, d_H)$};
  \node (M) at (2.2,0) {$ (C_{\bullet}, d_C)$};

  \draw[->] ([yshift=3pt]L.east) -- node[above] {$i$} ([yshift=3pt]M.west);
  \draw[->] ([yshift=-3pt]M.west) -- node[below] {$p$} ([yshift=-3pt]L.east);

  \node (H) at (3.0,0) {};
  \draw[->] (H) edge[loop right, min distance=10mm] node[right] {$h$} (H);
\end{tikzpicture}
\]
consists of two chain complexes $(H_{\bullet}, d_H)$ and $(C_{\bullet}, d_C)$, two chain maps $i$ and $p$, and a homotopy $h: C_{\bullet} \rightarrow C_{\bullet}[1]$ such that
\begin{itemize}
    \item $pi = \id_{H}$,
    \item $ip = \id_{C} + d_C h + h d_C$, i.e., $h$ is a homotopy between $ip$ and $\id_C$,
    \item $h^2 = 0$, $hi = 0$ and $ph = 0$.
\end{itemize}
\end{definition}

\begin{definition}
    A \textit{perturbation} of a SDR datum is a map $\delta: C_{\bullet} \rightarrow C_{\bullet}[-1]$ such that $d_C + \delta$ is a differential, and $h\delta$ is nilpotent.
\end{definition}

\begin{theorem}[Homological perturbation lemma]\label{thm:HPL}
    Let $(H_{\bullet}, C_{\bullet}, i, p, h)$ be a strong deformation retract datum and $\delta: C_{\bullet} \rightarrow C_{\bullet}[-1]$ a perturbation. We define
    \begin{itemize}
        \item $\Sigma = \sum_{i=0}^{\infty} \delta(h\delta)^i$,
        \item $d_{H,\infty}=d_H + p\Sigma i$,
        \item $d_{C,\infty} = d_C + \delta,$
        \item $i_{\infty} = i + h\Sigma i$,
        \item $p_{\infty} = p + p\Sigma h$,
        \item $h_{\infty} = h + h\Sigma h$.
    \end{itemize}
    Then the following datum
        \[
\begin{tikzpicture}[>=Stealth,baseline=(current bounding box.center)]
  \node (L) at (0,0) {$ (H_{\bullet}, d_{H,\infty})$};
  \node (M) at (3.2,0) {$ (C_{\bullet}, d_{C,\infty})$};

  \draw[->] ([yshift=3pt]L.east) -- node[above] {$i_{\infty}$} ([yshift=3pt]M.west);
  \draw[->] ([yshift=-3pt]M.west) -- node[below] {$p_{\infty}$} ([yshift=-3pt]L.east);

  \node (H) at (4.2,0) {};
  \draw[->] (H) edge[loop right, min distance=10mm] node[right] {$h_{\infty}$} (H);
\end{tikzpicture}
\]
    is again a strong deformation retract datum. In particular, the complex $(H_{\bullet}, d_{H,\infty})$ is homotopy equivalent to the complex $(C_{\bullet}, d_C + \delta)$.
\end{theorem}
We need the following lemma before giving the proof of this theorem.
\begin{lemma}
    We have identities
    \begin{align*}
        \delta h \Sigma = \Sigma h \delta = \Sigma - \delta,\quad (1-\delta h)^{-1} = 1+\Sigma h,\quad (1-h\delta)^{-1} = 1 + h\Sigma
    \end{align*}
    and
    \begin{align*}
        \Sigma i p \Sigma + \Sigma d + d \Sigma = 0.
    \end{align*}
\end{lemma}
\begin{proof}
    The first three identities follow immediately from the formal power expansion of $\Sigma$. Now we have
    \begin{align*}
        \Sigma & i p \Sigma + \Sigma d + d \Sigma  = \Sigma(1+ d h + h d) \Sigma + \Sigma d + d \Sigma \\
        & = \Sigma^2 + \Sigma d(h\Sigma + 1) + (1+\Sigma h)d\Sigma \\
        & = \Sigma^2 + \Sigma d(1-h\delta)^{-1} + (1-\delta h)^{-1}d\Sigma \\
        & = (1-\delta h)^{-1}\left( (1-\delta h) \Sigma^2 (1-h\delta) + (1-\delta h)\Sigma d + d\Sigma (1-h\delta) \right) (1-h\delta)^{-1}\\
        & = (1-\delta h)^{-1}\left( (\Sigma-\delta h \Sigma) (\Sigma- \Sigma h\delta) + (\Sigma-\delta h \Sigma) d + d (\Sigma-\Sigma h\delta) \right) (1-h\delta)^{-1} \\
        & = (1-\delta h)^{-1}\left( \delta^2 + \delta d + d \delta \right) (1-h\delta)^{-1}
    \end{align*}
    Finally note that $\delta^2 + \delta d + d \delta = 0$ since $d+\delta$ is a differential, and the lemma follows.
\end{proof}

\begin{proof}[Proof of Theorem \ref{thm:HPL}]
    The only nontrivial part of the proof is to check $$i_{\infty}p_{\infty} = \id_{C} + d_{C,\infty} h_{\infty} + h_{\infty} d_{C,\infty}.$$ 
    Expanding $\id_{C} + d_{C,\infty} h_{\infty} + h_{\infty} d_{C,\infty} - i_{\infty}p_{\infty}$ we have
    \begin{align*}
        \id_C + &(d_C + \delta)(h + h\Sigma h) + (h + h\Sigma h)(d_C + \delta) - (i + h\Sigma i)(p + p\Sigma h)  = \\
        &= \id_C + d_C h + d_C h \Sigma h + \delta h + \delta h \Sigma h + h d_C + h \delta \\
        &\qquad + h \Sigma h d_C + h \Sigma h \delta - ip - ip\Sigma h - h\Sigma ip - h\Sigma ip \Sigma h 
    \end{align*}
    Using $\Sigma h\delta = \Sigma - \delta$, $\delta h \Sigma = \Sigma - \delta$, $ip = \id_C +d_C h + h d_C$ and $\Sigma ip \Sigma = -(\Sigma d_C + d_C \Sigma)$, we get
    \begin{align*}
        \id_{C} + & d_{C,\infty} h_{\infty} + h_{\infty} d_{C,\infty} - i_{\infty}p_{\infty} \\
        &= h \delta + h\Sigma h d_C +  h(\Sigma-\delta) + \delta h + d_C h \Sigma h + (\Sigma - \delta)h \\
        &\qquad\quad - ip\Sigma h - h\Sigma ip + h(\Sigma d_C + d_C \Sigma)h \\
        &=h\Sigma h d_C + h\Sigma + d_C h \Sigma h + \Sigma h - ip \Sigma h - h\Sigma ip + h\Sigma d_C h + h d_C\Sigma h \\
        &=h\Sigma(hd_C + \id_C - ip + d_C h) + (d_C h + \id_C - ip + h d_C)\Sigma h
    \end{align*}
    Finally again by $ip = \id_C + d_C h + h d_C$ this is equal to zero.
\end{proof}

\section{Construction of minimal resolutions}\label{sec:construction}

In this section we apply the homological perturbation lemma to construct a strong deformation retract of the HHL resolution. We will then show that the resulting complex is indeed a minimal resolution. The combinatorial description of the differentials will be given in the next section.

\smallskip

We begin this section by discussing the choice of the homotopy map \(h\) used in the homological perturbation lemma. In the construction below, we use the Moore--Penrose inverse to construct such a homotopy. In principle, there are many other possible choices of \(h\). As long as \(h\) defines a strong deformation retract of the associated graded complex of the HHL resolution onto its homology complex, the arguments of this section go through. See Remark~\ref{rmk:choice-of-homotopy-2} and the example at the end of Section~\ref{sec:example} for more details.

\smallskip

\begin{definition}
    We endow the Bondal-Thomsen collection $BT$ with the partial order $\leq$ defined by
    \begin{align*}
    \mathcal{O}(\tau) \leq \mathcal{O}(\sigma) \text{ if and only if } \operatorname{Hom}(\mathcal{O}(\sigma), \mathcal{O}(\tau)) \neq 0,
\end{align*}
and call $(BT,\leq)$ the \textit{Bondal-Thomsen poset}.
\end{definition}

The HHL resolution $(C_{\bullet},d)$ admits a natural filtration by the Bondal-Thomsen poset.

\begin{definition}
    Fix $[a]\in BT$. We define $F_{[a]}C_{\bullet}$ to be
    \begin{align*}
        0\rightarrow\bigoplus_{\substack{\dim\sigma=k \\ \mathcal{O}(\sigma)\leq [a] }} \mathcal{O}(\sigma) \rightarrow \cdots \rightarrow \bigoplus_{\substack{\dim\sigma=1 \\ \mathcal{O}(\sigma)\leq [a]}} \mathcal{O}(\sigma) \rightarrow \bigoplus_{\substack{\dim\sigma=0 \\ \mathcal{O}(\sigma)\leq [a]}} \mathcal{O}(\sigma)\rightarrow0
    \end{align*}
    i.e., the subcomplex of the HHL resolution where only $\mathcal{O}(\sigma)$ with $\mathcal{O}(\sigma)\leq [a]$ appear. Then $\{F_{[a]}C_{\bullet}\}_{[a]\in BT}$ defines a filtration on the HHL complex by the Bondal-Thomsen poset. Furthermore, we denote the associated graded complex by 
    \begin{align*}
        \gr^F C_{\bullet} = \bigoplus_{[a]} \gr^F_{[a]}C_{\bullet} = \bigoplus_{[a]}\left(F_{[a]}C_{\bullet}/F_{<[a]}C_{\bullet}\right).
    \end{align*}
\end{definition}
\begin{remark}
    The associated graded complex $\gr^F C_{\bullet}$ is easily seen to be the complex obtained by setting all variables $x_i$ to zero in the HHL resolution. In particular, all entries in the differentials $\gr^F d$ are either $0$ or $\pm1$.
\end{remark}

Recall that for a fixed $[a]\in BT$, we denote by $S_{[a]}$ the set of strata $\sigma$ in the Bondal stratification such that $\mathcal{O}(\sigma)$ is isomorphic to $[a]$. We consider the locally finite cellular chain complex
\begin{align*}
    C^{\mathrm{BM}}_{\bullet}(S_{[a]}):\ 0\rightarrow \bigoplus_{\substack{\sigma\in S_{[a]} \\ \dim\sigma = k}}\C\sigma \xrightarrow{\partial_{[a]}} \cdots \rightarrow \bigoplus_{\substack{\sigma\in S_{[a]} \\ \dim\sigma = 1}}\C\sigma \xrightarrow{\partial_{[a]}} \bigoplus_{\substack{\sigma\in S_{[a]} \\ \dim\sigma = 0}}\C\sigma \rightarrow 0
\end{align*}
that computes the Borel-Moore homology (with $\C$-coefficients) of the stratum $S_{[a]}$. If we endow each $\C$-vector space with the standard Hermitian inner product by declaring the cell generators to be an orthonormal basis, the differentials $\partial_{[a]}$ and their adjoints $\partial_{[a]}^*$ (which are just the transposes since all matrices are over $\Z$) define the following combinatorial Hodge decomposition
\begin{align}\label{eq:comb-hodge-decomposition}
    \bigoplus_{\substack{\sigma\in S_{[a]} \\ \dim\sigma = i}}\C\sigma = \im\,\partial_{[a],i+1}\oplus \mathcal{H}_{[a],i} \oplus \im\,\partial_{[a],i}^*
\end{align}
that induces a canonical isomorphism between the homology $H_i^{\mathrm{BM}}(S_{[a]})$ and the harmonic part $\mathcal{H}_{[a],i}$. Therefore we have an embedding $i: H^{\mathrm{BM}}_{\bullet}(S_{[a]})\hookrightarrow C^{\mathrm{BM}}_{\bullet}(S_{[a]})$ and a projection $p: C^{\mathrm{BM}}_{\bullet}(S_{[a]}) \twoheadrightarrow H^{\mathrm{BM}}_{\bullet}(S_{[a]})$.

\smallskip

Next we use the Moore-Penrose inverses of the differentials of $C^{\mathrm{BM}}_{\bullet}(S_{[a]})$ to construct a homotopy map that induces a strong deformation retract onto its homology complex.

\begin{lemma}
    We have a SDR datum
    \[
\begin{tikzpicture}[>=Stealth,baseline=(current bounding box.center)]
  \node (L) at (0,0) {$ (H^{\mathrm{BM}}_{\bullet}(S_{[a]}), 0)$};
  \node (M) at (4.2,0) {$ (C^{\mathrm{BM}}_{\bullet}(S_{[a]}), \partial_{[a]})$};

  \draw[->] ([yshift=3pt]L.east) -- node[above] {$i$} ([yshift=3pt]M.west);
  \draw[->] ([yshift=-3pt]M.west) -- node[below] {$p$} ([yshift=-3pt]L.east);

  \node (H) at (5.5,0) {};
  \draw[->] (H) edge[loop right, min distance=10mm] node[right] {$-\partial_{[a]}^+$} (H);
\end{tikzpicture}
\]
between $C^{\mathrm{BM}}_{\bullet}(S_{[a]})$ and its homology complex $H^{\mathrm{BM}}_{\bullet}(S_{[a]})$ with trivial differentials.
\end{lemma}
\begin{proof}
Now the Moore-Penrose inverse $\partial_{[a],i-1}^+$ kills $\im\,\partial_{[a],i-1}^*$ and $\mathcal{H}_{[a],i-1}$, and restricts to an isomorphism between $\im\,\partial_{[a],i}$ and $\im\,\partial_{[a],i}^*$. By definition of the Moore-Penrose inverse, the isomorphisms
\begin{align*}
    \partial_{[a],i}: \im\,\partial_{[a],i}^*\xrightarrow{\sim} \im\,\partial_{[a],i-1},\qquad \partial_{[a],i-1}^+: \im\,\partial_{[a],i-1}\xrightarrow{\sim} \im\,\partial_{[a],i}^*
\end{align*}
are inverse to each other.

\smallskip

By the definitions of $i$ and $p$ it is clear that $pi=\id_H$. It is also clear that $\partial_{[a],i}\partial_{[a],i}^{+}$ is the projection onto $\operatorname{im}\partial_{[a],i}$, while $\partial_{[a],i}^{+}\partial_{[a],i}$ is the projection onto $\operatorname{im}\partial^*_{[a],i}$. Therefore we have $ip = \id_C - \partial_{[a],i}\partial_{[a],i-1}^+ - \partial_{[a],i-1}^+\partial_{[a],i}$. Finally $(\partial_{[a]}^+)^2=0$, $\partial_{[a]}^+ i =0$ and $p\partial_{[a]}^+ = 0$ because of the action of $\partial_{[a]}^+$ on the decomposition \eqref{eq:comb-hodge-decomposition}. Therefore we get a SDR datum as claimed.
\end{proof}

\smallskip

It is then straightforward to lift these results to the associated graded complex of the HHL resolution.

\smallskip

\begin{lemma}
    For each $[a]\in BT$, the graded piece $\gr^F_{[a]}C_{\bullet} = F_{[a]}C_{\bullet}/F_{<[a]}C_{\bullet}$ is isomorphic to $C^{\mathrm{BM}}(S_{[a]})\otimes_{\C}\mathcal{O}(a)$. In particular, the homology $H_i(\gr^F_{[a]}C_{\bullet})$ is isomorphic to $H^{\mathrm{BM}}_{i}(S_{[a]})\otimes\mathcal{O}(a)$.
\end{lemma}
\begin{proof}
    This is straightforward because all line bundles in the graded piece $\gr^F_{[a]}C_{\bullet}$ are isomorphic to $[a]$, and the differentials in $\gr^F_{[a]}C_{\bullet}$ are exactly the same as the differentials in $C^{\mathrm{BM}}(S_{[a]})$.
\end{proof}

\begin{corollary}
    We have the following SDR datum 
    \[
\begin{tikzpicture}[>=Stealth,baseline=(current bounding box.center)]
  \node (L) at (0,0) {$ (H_{\bullet}, 0)$};
  \node (M) at (3.2,0) {$ (\gr^F C_{\bullet}, \gr^F d)$};

  \draw[->] ([yshift=3pt]L.east) -- node[above] {$i$} ([yshift=3pt]M.west);
  \draw[->] ([yshift=-3pt]M.west) -- node[below] {$p$} ([yshift=-3pt]L.east);

  \node (H) at (4.5,0) {};
  \draw[->] (H) edge[loop right, min distance=10mm] node[right] {$h$} (H);
\end{tikzpicture}
\]
where $(H_{\bullet},0)$ is the homology complex of the associated graded complex of the HHL resolution with trivial differentials
\begin{align*}
    0\rightarrow\bigoplus_{\substack{[a]\in BT }} \mathcal{H}_{[a],k}\otimes \mathcal{O}_{\mathcal{X}}(a) \xrightarrow{0} \cdots  \xrightarrow{0} \bigoplus_{\substack{[a]\in BT }} \mathcal{H}_{[a],0}\otimes\mathcal{O}_{\mathcal{X}}(a)\rightarrow0
\end{align*}
and the homotopy $h$ is defined as $\bigoplus_{[a]\in BT} -\partial_{[a]}^+(-)\otimes \mathcal{O}(a)$.
\end{corollary}
\begin{proof}
    It follows immediately from the previous two lemmas by tensoring with $\mathcal{O}(a)$ and taking direct sum over all $[a]\in BT$.
\end{proof}

Now we define $\delta = d - \gr^F d$ as the difference between the original differentials on the HHL resolution $(C_{\bullet},d)$ and the differentials on the associated graded complex $(\gr^F C_{\bullet}, \gr^F d)$. The following lemma shows that $\delta$ defines a perturbation on the latter.

\begin{lemma}
    The composition $h\delta$ is nilpotent.
\end{lemma}
\begin{proof}
    By definition of $\delta$ the image of $\mathcal{O}(\sigma)\in BT$ is contained in the direct sum of $\mathcal{O}(\tau)$ with $\mathcal{O}(\tau) < \mathcal{O}(\sigma)$ in the Bondal-Thomsen poset. Note that by definition of $h$ the image of $\mathcal{O}(\tau)$ is contained in the direct sum of $\mathcal{O}(\tau')$ with $\mathcal{O}(\tau')= \mathcal{O}(\tau)$. Therefore $h\delta$ strictly lowers the position in the Bondal-Thomsen poset, which implies that $h\delta$ is nilpotent since $BT$ is finite.
\end{proof}

\smallskip

Therefore by the homological perturbation lemma we obtain a new SDR datum
\[
\begin{tikzpicture}[>=Stealth,baseline=(current bounding box.center)]
  \node (L) at (0,0) {$ (H_{\bullet}, d_{\min})$};
  \node (M) at (3.2,0) {$ (C_{\bullet}, d)$};

  \draw[->] ([yshift=3pt]L.east) -- node[above] {$i_{\infty}$} ([yshift=3pt]M.west);
  \draw[->] ([yshift=-3pt]M.west) -- node[below] {$p_{\infty}$} ([yshift=-3pt]L.east);

  \node (H) at (4.0,0) {};
  \draw[->] (H) edge[loop right, min distance=10mm] node[right] {$h_{\infty}$} (H);
\end{tikzpicture}
\]
where $\Sigma = \sum_{i=0}^{\infty} \delta(h\delta)^i$ and $d_{\min}=p\Sigma i$.

\smallskip

\begin{theorem}\label{thm:main}
    The complex $(H_{\bullet}, d_{\min})$:
    \begin{align*}
    0\rightarrow\bigoplus_{\substack{[a]\in BT }} \mathcal{H}_{[a],k}\otimes \mathcal{O}_{\mathcal{X}}(a) \xrightarrow{d_{\min}} \cdots  \xrightarrow{d_{\min}} \bigoplus_{\substack{[a]\in BT }} \mathcal{H}_{[a],0}\otimes\mathcal{O}_{\mathcal{X}}(a)\rightarrow0
\end{align*}
is a minimal resolution of $i_*\mathcal{O}_{\mathcal{Y}}$ by line bundles in the Bondal-Thomsen collection.
\end{theorem}
\begin{proof}
    By the homological perturbation lemma, $(H_{\bullet}, d_{\min})$ has the same homology as the HHL resolution $(C_{\bullet},d)$. Hence, by Theorem \ref{thm:HPL}, $(H_{\bullet}, d_{\min})$ is a resolution of the pushforward of the structure sheaf of the toric substack. Now note that by definition $\Sigma$ sends each $\mathcal{O}(\sigma)$ to some $\mathcal{O}(\tau)$ with $\mathcal{O}(\tau)<\mathcal{O}(\sigma)$, therefore none of its matrix entries is invertible. Therefore $H_{\bullet}$ is minimal.
\end{proof}

\smallskip

\begin{remark}\label{rmk-projection-onto-harmonic}
    If we think of each term of $H_{\bullet}$ as the harmonic subspace of the corresponding term in the HHL resolution $C_{\bullet}$, then we can simply ignore the embedding $i$ in the differential $d_{\min}=p\Sigma i$, and the projection $p$ is defined by $\id - dd^+ - d^+ d$. Therefore the differential is given by 
    \begin{align*}
        d_{\min} = (\id + dh + h d)\Sigma
    \end{align*}
    It is worth mentioning, however, that the embedding $i_{\infty}$ and projection $p_{\infty}$ between the minimal resolution and the HHL resolution as morphisms between complexes are different from the usual embedding and projection $i$ and $p$ defined by the Hodge decomposition. In principle it is possible to express these morphisms also in combinatorial terms, but we will not focus on this issue because we are only interested in the minimal resolution $H_{\bullet}$ as an abstract complex rather than a certain subcomplex of the HHL resolution. To emphasize this point, from now on we will denote $H_{\bullet}$ by 
    \begin{align*}
    0\rightarrow\bigoplus_{\substack{[a]\in BT }} \mathcal{O}_{\mathcal{X}}(a)^{\oplus\beta_{k,a}} \xrightarrow{d_{\min}} \cdots  \xrightarrow{d_{\min}} \bigoplus_{\substack{[a]\in BT }} \mathcal{O}_{\mathcal{X}}(a)^{\oplus\beta_{0,a}}\rightarrow0
\end{align*}
where $\beta_{i,a} = \dim H^{\mathrm{BM}}_{i}(S_{[a]},\C)$.
\end{remark}

\section{Combinatorial description of the differentials}\label{sec:differentials}

In this section we will give a combinatorial description of the differentials $d_{\min}$ of the minimal resolution. To achieve this goal, we first describe the morphism $\Sigma$ on the HHL resolution in terms of certain weighted sums over paths in the Bondal stratification.

\smallskip

\begin{definition}
    Let $(\sigma,\tau)$ be a pair of cells in the Bondal stratification $S$ with $\dim\sigma = \dim\tau + 1$ and $\mathcal{O}(\sigma)\geq \mathcal{O}(\tau)$ in the Bondal-Thomsen poset $BT$. We say that $(\sigma,\tau)$ is \textit{of type I} if $\mathcal{O}(\sigma)> \mathcal{O}(\tau)$, and \textit{of type II} if $\mathcal{O}(\sigma)= \mathcal{O}(\tau)$.
\end{definition}

\smallskip

In general, the morphism $\mathcal{O}(\sigma)\xrightarrow{\Sigma}\mathcal{O}(\tau)$ will be defined in terms of zig-zag paths that consist of alternating type I and type II pairs.

\smallskip

\begin{definition}\label{def:path}
    Let $(\sigma,\tau)$ be a pair of cells. We define a \textit{path} $\lambda$ from $\sigma$ to $\tau$ to be a sequence of cells
    
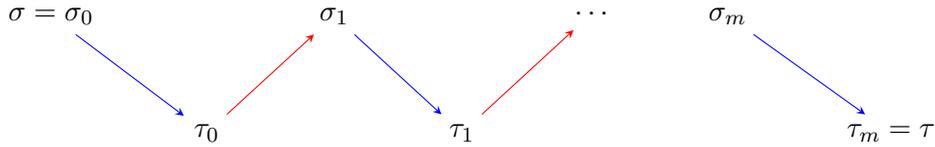
\[
\begin{tikzpicture}[>=stealth]
\matrix (m) [matrix of math nodes,
  row sep=2.8em,
  column sep=2.8em,
  ampersand replacement=\&]{
  \sigma=\sigma_0 \& \& \sigma_1 \& \& \cdots \& \sigma_m \& \\
  \& \tau_0 \& \& \tau_1 \& \& \& \tau_m=\tau \\
};

\draw[->, blue] (m-1-1) -- (m-2-2);
\draw[->,red] (m-2-2) -- (m-1-3);
\draw[->, blue] (m-1-3) -- (m-2-4);
\draw[->,red] (m-2-4) -- (m-1-5);
\draw[->, blue] (m-1-6) --  (m-2-7);
\end{tikzpicture}
\]
    such that:
    \begin{itemize}
        \item each $(\sigma_i,\tau_i)$ is of type I (blue),
        \item each $(\sigma_{i+1},\tau_i)$ is of type II (red).
    \end{itemize}
    We define the length $l(\lambda)$ of the path $\lambda$ to be the number of type II pairs in $\lambda$. We denote the set of all paths from $\sigma$ to $\tau$ by $\mathcal{P}(\sigma,\tau)$.
\end{definition}

\begin{definition}[HHL weights]
    Let $(\sigma,\tau)$ be a type I pair. The weight of $(\sigma,\tau)$ is defined as the $(\sigma,\tau)$-entry of the differential in the HHL resolution, i.e., the component corresponding to $\mathcal{O}(\sigma)\to \mathcal{O}(\tau)$, which we denote by $\mathrm{HHL}_{\sigma,\tau}$ and we call this the \textit{HHL weight} associated to the pair.
\end{definition}

\smallskip

Now we define weights for type II pairs. First recall that for each complex $\left(C^{\mathrm{BM}}_{\bullet}(S_{[a]}),\partial_{[a]}\right)$, we have the Moore-Penrose inverse $\partial_{[a]}^+$. For a type II pair $(\sigma,\tau)$, we define its weight to be the coefficient of $\sigma$ in $\partial_{[a]}^+(\tau)$. The combinatorial nature of Moore-Penrose inverses described in Section~\ref{subsec:MP-inverses} allows us to describe these weights in purely combinatorial terms. In this setting, the hedges are given by the following data:

\smallskip

\begin{definition}
    For each $[a]\in BT$, a \textit{hedge} $ST_{[a]}$ is a collection $\{T_i, S_{i-1}\}_{i=1}^k$, where $T_i \subseteq S_{[a]}^i$ and $S_{i-1} \subseteq S_{[a]}^{i-1}$ such that the compositions
$$ \bigoplus_{\sigma\in T_i}\C\sigma \hookrightarrow \bigoplus_{\substack{\sigma\in S_{[a]} \\ \dim\sigma = i}} \C\sigma \twoheadrightarrow \im\,\partial_{[a],i}$$
and
$$ \im\,\partial_{[a],i}\hookrightarrow \bigoplus_{\substack{\tau\in S_{[a]} \\ \dim\tau = i-1}} \C\tau \twoheadrightarrow \bigoplus_{\tau\in S_{i-1} }\C\tau  $$
are isomorphisms.
\end{definition}

Applying Proposition \ref{prop:MP-inverse-entry}, we have the following definition.

\begin{definition}[MP weights]\label{def:MP-weights}
    Let $(\sigma,\tau)$ be a type II pair. We define the \textit{MP weight} associated to the pair $(\sigma,\tau)$ as the coefficient of $\sigma$ in $\partial_{[a]}^+(\tau)$, explicitly given by the formula\footnote{Note that the complex conjugate is dropped because all entries are integers.}
    \begin{align*}
    \mathrm{MP}_{\sigma,\tau}:=\frac{1}{\Delta}\sum_{\substack{ ST \text{ hedge} \\ \tau\in S, \sigma\in T}} (-1)^{\mathrm{pos}_T(\sigma)+\mathrm{pos}_S(\tau)} \det(\partial_{[a],i,S\times T})\det(\partial_{[a],i,S\backslash\{\tau\}\times T\backslash\{\sigma\}}).
\end{align*}
where the sum is over all hedges that contain $\tau$ in the stake set and $\sigma$ in the shrubbery, and $\mathrm{pos}_T(\sigma)$ and $\mathrm{pos}_S(\tau)$ are the positions of $\sigma$ in $T$ and $\tau$ in $S$ respectively.
\end{definition}

\begin{definition}[Weight of paths]
    For any path $\lambda\in\mathcal{P}(\sigma,\tau)$ from $\sigma$ to $\tau$ as defined in Definition \ref{def:path}:
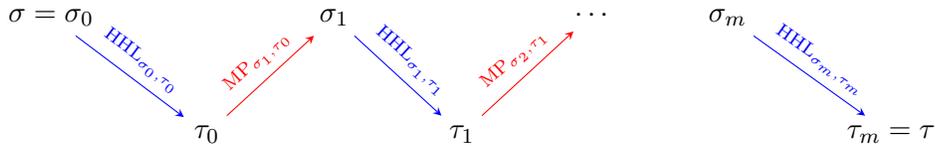
    \[
\begin{tikzpicture}[>=stealth]
\matrix (m) [matrix of math nodes,
  row sep=2.8em,
  column sep=2.8em,
  ampersand replacement=\&]{
  \sigma=\sigma_0 \& \& \sigma_1 \& \& \cdots \& \sigma_m \& \\
  \& \tau_0 \& \& \tau_1 \& \& \& \tau_m=\tau \\
};

\draw[->, blue] (m-1-1) -- node[midway, above, sloped, font=\tiny]
  {\color{blue}$\mathrm{HHL}_{\sigma_0,\tau_0}$} (m-2-2);
\draw[->,red] (m-2-2) -- node[midway, above, sloped, font=\tiny]
  {\color{red}$\mathrm{MP}_{\sigma_1,\tau_0}$} (m-1-3);
\draw[->, blue] (m-1-3) -- node[midway, above, sloped, font=\tiny]
  {\color{blue}$\mathrm{HHL}_{\sigma_1,\tau_1}$} (m-2-4);
\draw[->,red] (m-2-4) -- node[midway, above, sloped, font=\tiny]
  {\color{red}$\mathrm{MP}_{\sigma_2,\tau_1}$} (m-1-5);
\draw[->, blue] (m-1-6) -- node[midway, above, sloped, font=\tiny]
  {\color{blue}$\mathrm{HHL}_{\sigma_m,\tau_m}$} (m-2-7);
\end{tikzpicture}
\]  
    we define its weight $w_{\lambda}$ as the product of the HHL weights and the MP weights along the path.
\end{definition}

The morphism $\mathcal{O}(\sigma)\rightarrow\mathcal{O}(\tau)$ is then defined as
\begin{align*}
    \Sigma_{\sigma,\tau}:=\sum_{\lambda\in \mathcal{P}(\sigma,\tau)} (-1)^{l(\lambda)} w_{\lambda}
\end{align*}
where the sum is taken over all paths $\lambda$ from $\sigma$ to $\tau$.

\begin{theorem}
    The morphism $\Sigma: C_{\bullet} \rightarrow C_{\bullet}$ on the HHL resolution is given by the direct sum of $\Sigma_{\sigma,\tau}$ over all pairs of strata $(\sigma,\tau)$.
\end{theorem}
\begin{proof}
    This follows directly from the definition $\Sigma = \sum_{i=0}^{\infty} \delta(h\delta)^i$. In our setting, $h$ is given by the Moore-Penrose inverses, and $\delta$ is the part of the HHL differential that strictly decreases in the Bondal-Thomsen poset (i.e., the non-invertible part). Therefore, the component $\mathcal{O}(\sigma)\rightarrow\mathcal{O}(\tau)$ of $\Sigma$ is given by the sum of the products of the HHL weights and MP weights along all paths. The sign $(-1)^{l(\lambda)}$ comes from the fact that $h$ is defined as $-\partial^+$.
\end{proof}

Finally, according to the discussion in Remark \ref{rmk-projection-onto-harmonic}, the minimal resolution $(H_{\bullet},d_{\min})$ can be obtained by post-composing $\Sigma$ with $\id + d h + h d$ and then restricting to the harmonic part.

\begin{remark}\label{rmk:choice-of-homotopy-2}
    For a different choice of the homotopy $h$, one only needs to modify the definition of the weights of the type II pairs accordingly, and the same combinatorial description of the morphism $\Sigma$ in terms of paths still holds.
\end{remark}

\section{An example}\label{sec:example}

In this section we work out an explicit example. Consider the weighted projective plane $\PP(3,1,1)$ whose homogeneous coordinate ring is given by $\C[x,y,z]$ where the weight of $x$ is 3. The three rays in the fan of $\PP(3,1,1)$ are $(1,0)$, $(0,1)$ and $(-3,-1)$. Consider the point defined by the ideal $(y-z,x-y^3)$ embedded into $\PP(3,1,1)$. The corresponding Bondal stratification and the HHL resolution are given as follows:
\begin{center}
    \begin{minipage}[c]{0.32\textwidth}
        \centering
        \begin{tikzpicture}[scale=1.2, >=Latex]

  \draw[thick] (0,0) rectangle (3,3);

  \draw[thick, black] (0,0) -- (3,0);
  \draw[thick, black] (0,3) -- (3,3);
  \draw[thick, black] (0,0) -- (0,3);
  \draw[thick, black] (3,0) -- (3,3);

  \draw[thick,black] (0,3) -- (1,0);
  \draw[thick,black] (1,3) -- (2,0);
  \draw[thick,black] (2,3) -- (3,0);

  \fill (0,0) circle (1.5pt);
  \fill (3,0) circle (1.5pt);
  \fill (0,3) circle (1.5pt);
  \fill (3,3) circle (1.5pt);
  \fill (1,0) circle (1.5pt);
  \fill (2,0) circle (1.5pt);
  \fill (1,3) circle (1.5pt);
  \fill (2,3) circle (1.5pt);

  	\node at (0.45,0.45) {$F_1$};
	\node at (1,1.5) {$F_2$};
    \node at (2,1.5) {$F_3$};
    \node at (2.55,2.55) {$F_4$};

    \node at (-0.3,1.5) {$E_1$};
    \node at (0.4,1.3) {$E_2$};
    \node at (1.4,1.3) {$E_3$};
    \node at (2.4,1.3) {$E_4$};
    \node at (0.5,3.2) {$E_5$};
    \node at (1.5,3.2) {$E_6$};
    \node at (2.5,3.2) {$E_7$};

	\node at (-0.3,-0.3) {$V_1$};
    \node at (1,-0.3) {$V_2$};
    \node at (2,-0.3) {$V_3$};

\end{tikzpicture}
    \end{minipage}%
    \begin{minipage}[c]{0.78\textwidth}
        \centering
        \begin{align*}
    0 \rightarrow \bigoplus_{i=1}^4 \mathcal{O}(F_i) \xrightarrow{d_1} \bigoplus_{i=1}^7 \mathcal{O}(E_i) \xrightarrow{d_0} \bigoplus_{i=1}^3 \mathcal{O}(V_i) \rightarrow 0
\end{align*}
    \end{minipage}
\end{center}
where the differentials $d_1$ and $d_0$ are given by
\begin{align*}
    d_1 = \begin{pmatrix}
        -x & 0 & 0 & 1 \\
        -z & 1 & 0 & 0 \\
         0 & -z & 1 & 0 \\
          0 & 0 & -z & 1 \\
          -y  & 1 & 0 & 0 \\
          0  & -y & 1 & 0 \\
           0 & 0 & -y & 1 \\
    \end{pmatrix},
    d_0 = \begin{pmatrix}
        y-z & x & 0 & -y & -x & 0 & z \\
        0 & -y & 1 & 0 & z & -1 & 0 \\
        0 & 0 & -y & 1 & 0 & z & -1
    \end{pmatrix}
\end{align*}

\begin{figure}
    \centering
\begin{tikzpicture}[
  x=2.4cm,
  y=1.8cm,
  >=Stealth,
  class/.style={circle, fill=black, inner sep=1.8pt},
  every node/.style={font=\small}
]

  \draw[step=1, gray!35, thin] (0.5,0.5) grid (5.5,4.5);

  \draw[->] (0.5,0.5) -- (5.5,0.5) node[right] {$BT$};
  \draw[->] (0.5,0.5) -- (0.5,4.8) node[above] {$q$};

   \node[below] at (1,0.5) {$\mathcal{O}$};
   \node[below] at (2,0.5) {$\mathcal{O}(-1)$};
   \node[below] at (3,0.5) {$\mathcal{O}(-2)$};
   \node[below] at (4,0.5) {$\mathcal{O}(-3)$};
   \node[below] at (5,0.5) {$\mathcal{O}(-4)$};

  \foreach \y in {1,...,4}
    \node[left] at (0.5,\y) {\y};

  \node at (1,3) {$\mathcal{O}(V_1)$};
  \node at (2,4) {$\mathcal{O}(F_4)$};
  \node at (2,3) {$\bigoplus\limits_{i=1,4,7}\mathcal{O}(E_i)$};
  \node at (2,2) {$\mathcal{O}(V_3)$};
  \node at (3,3) {$\mathcal{O}(F_3)$};
  \node at (3,2) {$\bigoplus\limits_{i=3,6}\mathcal{O}(E_i)$};
  \node at (3,1) {$\mathcal{O}(V_2)$};
  \node at (4,2) {$\mathcal{O}(F_2)$};
    \node at (4,1) {$\bigoplus\limits_{i=2,5}\mathcal{O}(E_i)$};
    \node at (5,1) {$\mathcal{O}(F_1)$};

  \draw[->, thick] (2,3.8) -- (2,3.2);
  \draw[->, thick] (2,2.8) -- (2,2.2);
  \draw[->, thick] (3,2.8) -- (3,2.2);
  \draw[->, thick] (3,1.8) -- (3,1.2);
  \draw[->, thick] (4,1.8) -- (4,1.2);

  \node[font = \tiny] at (2.2,3.6) {$\begin{pmatrix}
    1 \\ 1 \\ 1
  \end{pmatrix}$};
  \node[font = \tiny] at (2.3,2.6) {$(0,1,-1)$};

  \node[font = \tiny] at (3.2,2.6) {$\begin{pmatrix}
    1 \\ 1 
  \end{pmatrix}$};
  \node[font = \tiny] at (3.3,1.6) {$(1,-1)$};

  \node[font = \tiny] at (4.2,1.6) {$\begin{pmatrix}
    1 \\ 1 
  \end{pmatrix}$};

  \draw[->,thick,red] (1.9,3.25) to[bend left] (1.9,3.8);
  \draw[->,thick,red] (1.9,2.25) to[bend left] (1.9,2.8);
  \draw[->,thick,red] (2.9,2.25) to[bend left] (2.9,2.8);
  \draw[->,thick,red] (2.9,1.25) to[bend left] (2.9,1.8);
  \draw[->,thick,red] (3.9,1.25) to[bend left] (3.9,1.8);

  \node[font = \tiny, red] at (1.55,3.5) {$(\frac{1}{3},\frac{1}{3},\frac{1}{3})$};
  \node[font = \tiny, red] at (1.6,2.5) {$\begin{pmatrix}
    0 \\ 1/2 \\ -1/2 
  \end{pmatrix}$};
  \node[font = \tiny, red] at (2.65,2.4) {$(\frac{1}{2},\frac{1}{2})$};
  \node[font = \tiny, red] at (2.55,1.5) {$\begin{pmatrix}
     1/2 \\ -1/2 
  \end{pmatrix}$};
  \node[font = \tiny, red] at (3.65,1.4) {$(\frac{1}{2},\frac{1}{2})$};

  \draw (0.75,3.15) rectangle (1.25,2.85);
  \draw (1.55,3.25) rectangle (2.45,2.8);
  \draw (3.6,1.2) rectangle (4.4,0.85);
  \draw (4.75,1.15) rectangle (5.25,0.85);

\end{tikzpicture}
\caption{$\partial$ (black) and the Moore-Penrose inverse $\partial^+$ (red)}
\label{fig:MP-inverse}
\end{figure}

The Bondal-Thomsen collection is easily seen to be
\begin{align*}
    BT=\{\mathcal{O},\mathcal{O}(-1),\mathcal{O}(-2),\mathcal{O}(-3),\mathcal{O}(-4)\}
\end{align*}
and the stratification is given by
\begin{align*}
    S_0=\{V_1\},\quad S_{-1}&=\{V_3,E_1,E_4,E_7, F_4\} \\
    S_{-2} = \{V_2, E_3, E_6, F_3\},\quad &S_{-3} = \{E_2, E_5, F_2\},\quad S_{-4} = \{F_1\}.
\end{align*}

The nonzero Borel-Moore homology groups are indicated by boxes in Figure \ref{fig:MP-inverse}, and each of them is 1-dimensional. Therefore it suffices to look at $\Sigma_{\sigma,\tau}$ where $(\sigma,\tau)$ is one of the following:
\begin{align*}
    (F_1, E_i),\ (E_i, V_1),\ i=1,2,4,5,7
\end{align*}
We take $(F_1, E_i)$ for $i=1,4,7$ as examples.

\smallskip

{\bf Paths from $F_1$ to $E_1$:} It is easy to see that there is only one path $F_1 \xrightarrow{\color{blue} -x}E_1$.

\smallskip

{\bf Paths from $F_1$ to $E_4$:} There are 4 different paths from $F_1$ to $E_4$:
\begin{align*}
    F_1\xrightarrow{\color{blue} -z} E_2\xrightarrow{\color{red}\frac{1}{2}} F_2\xrightarrow{\color{blue}-z} E_3\xrightarrow{\color{red}\frac{1}{2}} F_3\xrightarrow{\color{blue}-z} E_4 \\
    F_1\xrightarrow{\color{blue} -y} E_5\xrightarrow{\color{red}\frac{1}{2}} F_2\xrightarrow{\color{blue}-z} E_3\xrightarrow{\color{red}\frac{1}{2}} F_3\xrightarrow{\color{blue}-z} E_4 \\
    F_1\xrightarrow{\color{blue} -z} E_2\xrightarrow{\color{red}\frac{1}{2}} F_2\xrightarrow{\color{blue}-y} E_6\xrightarrow{\color{red}\frac{1}{2}} F_3\xrightarrow{\color{blue}-z} E_4 \\
    F_1\xrightarrow{\color{blue} -y} E_5\xrightarrow{\color{red}\frac{1}{2}} F_2\xrightarrow{\color{blue}-y} E_6\xrightarrow{\color{red}\frac{1}{2}} F_3\xrightarrow{\color{blue}-z} E_4 
\end{align*}
where the HHL weights are shown in blue and the MP weights are shown in red. We provide the details of the computation of the Moore-Penrose inverse of the differential in the following sequence:
\begin{align*}
    0\rightarrow\C\cdot F_2\xrightarrow{\begin{pmatrix} 1 \\ 1\end{pmatrix}} \C\cdot E_2\oplus \C\cdot E_5 \rightarrow0.
\end{align*}
The hedges consist of a shrubbery $T\subseteq\{F_2\}$ and a stake set $S\subseteq\{E_2,E_5\}$. There are 2 possible hedges for this sequence: $(T,S_2)$ and $(T,S_5)$ where $T=\{F_2\}$, $S_2=\{E_2\}$ and $S_5=\{E_5\}$. The corresponding hedge splittings are defined by
\begin{align*}
    \partial^+_{S_2 T} = (1,0),\qquad \partial^+_{S_5 T} = (0,1)
\end{align*}
Therefore the Moore-Penrose inverse of the differential $\partial$ is given by the average $(\frac{1}{2},\frac{1}{2})$. The other Moore-Penrose inverses are computed in a similar manner. Finally, taking the sum of these contributions and keeping track of the signs, we obtain $\mathcal{O}(F_1)\xrightarrow{-\frac{1}{4}z(y+z)^2}\mathcal{O}(E_4)$.

\smallskip

{\bf Paths from $F_1$ to $E_7$:} Similarly, there are 4 paths as in the $E_4$ case, just replacing the end of each path by $E_7$ and modifying the HHL weights of the last step accordingly. A similar computation yields $\mathcal{O}(F_1)\xrightarrow{-\frac{1}{4}y(y+z)^2}\mathcal{O}(E_7)$.

\medskip

Putting all these together we get
\begin{align*}
    \mathcal{O}(F_1)\xrightarrow{\begin{pmatrix}
    -x \\ -\frac{1}{4}z(y+z)^2 \\ -\frac{1}{4}y(y+z)^2
    \end{pmatrix}}\mathcal{O}(E_1)\oplus \mathcal{O}(E_4) \oplus \mathcal{O}(E_7)
\end{align*}
If we choose $(-2,1,1)$ as the harmonic representative of the 1-dimensional homology group at $\bigoplus_{i=1,4,7}\mathcal{O}(E_i)$, by projecting to homology we get
\begin{align*}
    \mathcal{O}(-4)\xrightarrow{2x-\frac{1}{4}(y+z)^3}\mathcal{O}(-1)
\end{align*}
Similarly, one can compute all other morphisms in the resolution.

\begin{figure}
    \centering
\begin{tikzpicture}[
  x=2.4cm,
  y=1.8cm,
  >=Stealth,
  class/.style={circle, fill=black, inner sep=1.8pt},
  every node/.style={font=\small}
]

  \draw[step=1, gray!35, thin] (0.5,0.5) grid (5.5,4.5);

  \draw[->] (0.5,0.5) -- (5.5,0.5) node[right] {$BT$};
  \draw[->] (0.5,0.5) -- (0.5,4.8) node[above] {$q$};

   \node[below] at (1,0.5) {$\mathcal{O}$};
   \node[below] at (2,0.5) {$\mathcal{O}(-1)$};
   \node[below] at (3,0.5) {$\mathcal{O}(-2)$};
   \node[below] at (4,0.5) {$\mathcal{O}(-3)$};
   \node[below] at (5,0.5) {$\mathcal{O}(-4)$};

  \node at (1,3) {$\mathcal{O}$};
  \node at (2,3) {$\mathcal{O}(-1)$};

    \node at (4,1) {$\mathcal{O}(-3)$};
    \node at (5,1) {$\mathcal{O}(-4)$};

    \draw[->, thick] (4.8,1.2) -- (2.2,2.8);
    \draw[->, thick] (3.8,1.2) -- (1.2,2.8);
    \draw[->, thick] (4.7,1) -- (4.2,1);
    \draw[->, thick] (1.7,3) -- (1.2,3);

    \node at (4.5, 0.85) {\color{red}$y-z$};
    \node at (1.5, 3.15) {\color{red}$y-z$};
    \node at (3.5, 2.4) {\color{red}$2x-\frac{(y+z)^3}{4}$};
    \node at (2.5, 1.5) {\color{red}$-2x+\frac{(y+z)^3}{4}$};

\end{tikzpicture}
\caption{Minimal resolution}
\label{fig:min-res}
\end{figure}

Eventually we get the minimal resolution
\begin{align*}
    0\rightarrow\mathcal{O}(-4)\xrightarrow{\begin{pmatrix} y-z \\ 2x-\frac{(y+z)^3}{4} \end{pmatrix}}\substack{ \mathcal{O}(-3) \\ \oplus \\ \mathcal{O}(-1) } \xrightarrow{\begin{pmatrix} -2x+\frac{(y+z)^3}{4} & y-z \end{pmatrix}} \mathcal{O}\rightarrow 0
\end{align*}
as shown in Figure \ref{fig:min-res}.

\medskip

We end this paper with a final remark. It is clear that the choice of the homotopy map $h$ on the associated graded complex of the HHL resolution is not unique. The advantage of the Moore-Penrose inverse is that it is canonical in the sense that it does not rely on any choices. Therefore the only input data needed for this construction is the Bondal stratification itself.

\smallskip

In general, if we are allowed to drop the canonicity requirement, then there are often other choices of $h$ that make the differentials simpler. In this example, if we use the homotopy induced by the perfect Morse matching
\begin{align*}
    V_2\leftrightarrow E_3,\quad V_3\leftrightarrow E_4,\quad
    E_5\leftrightarrow F_2,\quad
    E_6\leftrightarrow F_3,\quad
    E_7\leftrightarrow F_4,
\end{align*}
then the homotopy maps are indicated in Figure \ref{fig:morse-matching}.
\begin{figure}
    \centering
\begin{tikzpicture}[
  x=2.4cm,
  y=1.8cm,
  >=Stealth,
  class/.style={circle, fill=black, inner sep=1.8pt},
  every node/.style={font=\small}
]

  \draw[step=1, gray!35, thin] (0.5,0.5) grid (5.5,4.5);

  \draw[->] (0.5,0.5) -- (5.5,0.5) node[right] {$BT$};
  \draw[->] (0.5,0.5) -- (0.5,4.8) node[above] {$q$};

   \node[below] at (1,0.5) {$\mathcal{O}$};
   \node[below] at (2,0.5) {$\mathcal{O}(-1)$};
   \node[below] at (3,0.5) {$\mathcal{O}(-2)$};
   \node[below] at (4,0.5) {$\mathcal{O}(-3)$};
   \node[below] at (5,0.5) {$\mathcal{O}(-4)$};

  \foreach \y in {1,...,4}
    \node[left] at (0.5,\y) {\y};

  \node at (1,3) {$\mathcal{O}(V_1)$};
  \node at (2,4) {$\mathcal{O}(F_4)$};
  \node at (2,3) {$\bigoplus\limits_{i=1,4,7}\mathcal{O}(E_i)$};
  \node at (2,2) {$\mathcal{O}(V_3)$};
  \node at (3,3) {$\mathcal{O}(F_3)$};
  \node at (3,2) {$\bigoplus\limits_{i=3,6}\mathcal{O}(E_i)$};
  \node at (3,1) {$\mathcal{O}(V_2)$};
  \node at (4,2) {$\mathcal{O}(F_2)$};
    \node at (4,1) {$\bigoplus\limits_{i=2,5}\mathcal{O}(E_i)$};
    \node at (5,1) {$\mathcal{O}(F_1)$};

  \draw[->, thick] (2,3.8) -- (2,3.2);
  \draw[->, thick] (2,2.8) -- (2,2.2);
  \draw[->, thick] (3,2.8) -- (3,2.2);
  \draw[->, thick] (3,1.8) -- (3,1.2);
  \draw[->, thick] (4,1.8) -- (4,1.2);

  \node[font = \tiny] at (2.2,3.6) {$\begin{pmatrix}
    1 \\ 1 \\ 1
  \end{pmatrix}$};
  \node[font = \tiny] at (2.3,2.6) {$(0,1,-1)$};

  \node[font = \tiny] at (3.2,2.6) {$\begin{pmatrix}
    1 \\ 1 
  \end{pmatrix}$};
  \node[font = \tiny] at (3.3,1.6) {$(1,-1)$};

  \node[font = \tiny] at (4.2,1.6) {$\begin{pmatrix}
    1 \\ 1 
  \end{pmatrix}$};

  \draw[->,thick,red] (1.9,3.25) to[bend left] (1.9,3.8);
  \draw[->,thick,red] (1.9,2.25) to[bend left] (1.9,2.8);
  \draw[->,thick,red] (2.9,2.25) to[bend left] (2.9,2.8);
  \draw[->,thick,red] (2.9,1.25) to[bend left] (2.9,1.8);
  \draw[->,thick,red] (3.9,1.25) to[bend left] (3.9,1.8);

  \node[font = \tiny, red] at (1.55,3.5) {$(0,0,1)$};
  \node[font = \tiny, red] at (1.6,2.5) {$\begin{pmatrix}
    0 \\ 1 \\ 0 
  \end{pmatrix}$};
  \node[font = \tiny, red] at (2.65,2.4) {$(0,1)$};
  \node[font = \tiny, red] at (2.65,1.5) {$\begin{pmatrix}
     1 \\ 0 
  \end{pmatrix}$};
  \node[font = \tiny, red] at (3.65,1.4) {$(0,1)$};

\end{tikzpicture}
\caption{Homotopies from a perfect Morse matching}
\label{fig:morse-matching}
\end{figure}
A similar computation yields the more familiar form of the minimal resolution
\begin{align*}
    0\rightarrow\mathcal{O}(-4)\xrightarrow{\begin{pmatrix} y-z \\ -x+y^3 \end{pmatrix}}\substack{ \mathcal{O}(-3) \\ \oplus \\ \mathcal{O}(-1) } \xrightarrow{\begin{pmatrix} x-y^3 & y-z \end{pmatrix}} \mathcal{O}\rightarrow 0.
\end{align*}

\end{document}